\documentclass[11pt]{amsart}
\usepackage{amssymb}
\def\R{\mathbb R}
\def\T{\mathcal T}
\def\N{\mathbb N}
\def\C{\mathbb C}
\def\e{\varepsilon}
\newtheorem{theorem}{Theorem}
\newtheorem{proposition}{Proposition}
\newtheorem{corollary}{Corollary}
\newtheorem{lemma}{Lemma}
\begin{document}
\title{Averages on annuli of Eulidean space}
\author{Fran\c{c}ois Havard}
\author{Emmanuel Lesigne}
\address[Emmanuel Lesigne et Fran\c{c}ois Havard]{Laboratoire de Math\'ematiques et Physique Th\'eorique (UMR CNRS 6083)\\
F\'ed\'eration de Recherche Denis Poisson\\ Universit\'e
Fran\c{c}ois Rabelais\\ Parc de Grandmont, 37200 Tours, France}
\email{emmanuel.lesigne@lmpt.univ-tours.fr}

\begin{abstract} We study the range of validity of differentiation theorems and ergodic theorems for $\R^d$ actions, for averages on ``thick spheres'' of Euclidean space.
\end{abstract}
\maketitle
\centerline{\today}

\setcounter{tocdepth}{1}%%Only sections appear in Contents
\tableofcontents
  \maketitle
  \tableofcontents
  
\section*{Introduction}
The classical Lebesgue differentiation Theorem states that, for any locally integrable function $f$ on the $d$-dimensional Euclidean space $\R^d$,
$$
\lim_{r\to0}\frac1{|B_r|}\int_{B_r}f(x+t)\,\text{d}t=f(x)\quad\text{for almost all $x$}.
$$
Here $B_r$ denotes the ball centered at the origin and of radius $r>0$, and $|A|$ denotes the Euclidean volume (i.e. the Lebesgue measure) of the set $A$. The reference measure on $\R^d$ is Lebesgue measure.

The main fact which ensures this differentiation theorem is the Hardy-Littlewood maximal inequality, which can be written in the following form: for any non-negative integrable function $f$ on $\R^d$, 
$$
\left|\left\{x\in\R^d\mid\sup_{r>0}\frac1{|B_r|}\int_{B_r}f(x+t)\,\text{d}t>1\right\}\right|\leq C(d)\int_{\R^d}f(x)\,\text{d}x,
$$
where $C(d)$ is a constant depending only on the dimension $d$.

It is known since the works of Wiener that the maximal inequality is also the cornerstone of the following ergodic theorem. \\Let $(\Omega,\T,\mu)$ be a probability space and $(T_t)_{t\in\R^d}$ be a measurable and measure preserving action of the group $\R^d$ on this space. Then, for all integrable function $\phi$ on $(\Omega,\T,\mu)$,
$$
\lim_{r\to+\infty}\frac1{|B_r|}\int_{B_r}\phi(T_t\omega)\,\text{d}t\quad \text{exists for $\mu$-almost all $\omega$}.
$$
Moreover if the action is ergodic, this limit is $\displaystyle\int_\Omega\phi(\omega)\,\text{d}\mu(\omega).$
\\(Basic definitions are recalled at the end of Introduction)

These classical results have been extended in many different directions, and the literature on differentiation and ergodic theorems is extremely wide (among general reviews see \cite{St2}, \cite{Krengel}, \cite{etga} ...). Variations can be made on the type of averaging region, the type of acting group, the type of averages and the function spaces under consideration.

In the present article, we study averages on domains of the Euclidean space with spherical symmetry, more precisely annuli
$$
C_{r,e}:=\{x\in\R^d\mid r-e\leq\|x\|\leq r\},
$$
where $\|\cdot\|$ denotes the standard Euclidean norm and $r$, $e$ are positive real numbers with $e\leq r$.

We will consider annuli for which the thickness $e$ is a function of the radius $r$. Thus we assume, we are given a function $r\mapsto e(r)$ from $(0,+\infty)$ into itself, with $e(r)\leq r$. And we define the following averaging operators.

If $f$ is a locally integrable function on $\R^d$,
$$
M_rf(x):=\frac1{\left|C_{r,e(r)}\right|}\int_{C_{r,e(r)}}f(x+t)\,\text{d}t.
$$

If $\left(\Omega,\T,\mu,(T_t)_{t\in\R^d}\right)$ is a measurable and measure preserving action of the group $\R^d$ on a measure space, and if $\phi\in L^1(\mu)$,
$$
A_r\phi(\omega):=\frac1{\left|C_{r,e(r)}\right|}\int_{C_{r,e(r)}}\phi(T_t\omega)\,\text{d}t.
$$

It is not difficult to verify that the family $(A_r\phi)_{r>0}$ is $\mu$-almost everywhere well defined (see Proposition \ref{def}). Considering $\R^d$ acting on itself by translation, we have in particular that the family $(M_rf)_{r>0}$ is almost everywhere well defined.

We are interested in the validity of a differentiation theorem, describing the behaviour of $M_rf$ when $r$ goes to zero, and of an ergodic theorem, describing the behaviour of $A_r\phi$ when $r$ goes to infinity. We know that a maximal inequality is the cornerstone of each of these results, when one looks for pointwise convergence.

Not surprisingly, the case of dimension $d=1$ differs significantly from the case of higher dimensions. In the present article we will be mostly interested in the case of dimension $d\geq2$. In this case, our study is related to the study of averages on spheres. Although less classical and more difficult to study than averages on balls, the averages on spheres have been thoroughly studied by Stein (\cite{St1}, differentiation theorem, $d\geq3$), Bourgain (\cite{B}, differentiation theorem, $d=2$), Jones (\cite{J}, ergodic theorem, $d\geq3$) and Lacey (\cite{L}, ergodic theorem, $d=2$). (A survey and numerous references can be found in \cite{etga}.) These results are described in Section \ref{sphere}. In particular, we recall that the maximal inequality for spherical averages is valid under the $L^p$-integrability condition, with $p>\frac{d}{d-1}$. 

In Section \ref{facile}, we describe the first basic results concerning averages on annuli: almost everywhere existence of averages, and their convergence in the mean.

Section \ref{principal} contains the most original part of our study. We describe the range of $L^p$ spaces in which a maximal inequality is satisfied by the families of averages $(M_rf)$ and $(A_r\phi)$. The reader could expect that this range varies ``continuously'' with the choice of the thickness function $e$, from the condition $p\geq1$ corresponding to the case of balls, toward the condition $p>\frac{d}{d-1}$ corresponding to the case of spheres. This is not true, and we establish a dichotomy theorem: under some mild natural condition of regularity on the function $e$, either the averages $M_r$ and $A_r$ behave like averages on balls, or they behave like averages on spheres. As soon as the thickness of the annuli is asymptotically negligible with respect to the radius, there is no $L^{d/(d-1)}$ maximal inequality.

In Section \ref{dimun}, we present quickly what happens in dimension 1.

\medbreak
We thank Anthony Quas for helpful discussions, and the suggestion of a rescaling argument which plays a crucial role in the negation of the weak-$L^{d/(d-1)}$ maximal inequality.

\medbreak
For the reader who is not familiar with ergodic theory, we recall some basic definitions. \\
A {\it measurable and measure preserving action} of $\R^d$ on a measure space $(\Omega,\T,\mu)$ is given by a family $(T_t)_{t\in\R^d}$ of maps from $\Omega$ into itself such that \begin{itemize}\item The map $\R^d\times\Omega\to\Omega,\ (t,\omega)\mapsto T_t(\omega)$ is measurable;\item For all $t\in\R^d$ and all $A\in\T$, $\mu\left(T^{-1}(A)\right)=\mu(A)$;\item $T_0=\text{Id}_{\Omega}$ and, for all $t,s\in\R^d$, $T_{t+s}=T_t\circ T_s$.\end{itemize}
Moreover this action is {\it ergodic} if for $A\in\T$, 
$$
\left[\forall t\in\R^d,\ T_t(A)=A\right]\ \Longrightarrow\ \mu(A)=0 \text{ or }\mu(\Omega\setminus A)=0.
$$

\section{Averages on spheres}\label{sphere}
This section does not contain any original results, but the facts that we recall are necessary for the understanding of the remainder.

Let us denote by $\sigma_r$ the uniform probability on the sphere $S_r$ centered at the origin and of radius $r$, in the standard Euclidean space $\R^d$. If $f$ is a continuous function from $\R^d$ into $\C$, we denote by $M_r^{S}f(x)$ its mean value on the sphere centered at $x$ and of radius $r$.
$$
M_r^{S}f(x)=\int_{S_r}f(x+t)\,\text{d}\sigma_r(t)=\int_{S_1}f(x+rt)\,\text{d}\sigma_1(t).
$$
The following strong maximal inequality is due to Stein for $d\geq3$ and to Bourgain for $d=2$.
\begin{theorem} Let $d$ be an integer $\geq2$ and $p$ be a real number $>\frac{d}{d-1}$. There exists a positive constant $C(d,p)$ such that, for all continuous function $f$ on $\R^d$, with compact support,
\begin{equation}\label{max-in-s}
\int_{\R^d}\sup_{r>0}\left|M_r^Sf(x)\right|^p\,\text{d}x\leq C(d,p)\int_{\R^d}\left|f(x)\right|^p\,\text{d}x.
\end{equation}
\end{theorem}
For this theorem, we refer to original articles \cite{St1}, \cite{B} or textbooks \cite{St2}, \cite{etga}. It is not very difficult to see that the lower bound $d/(d-1)$ is optimal in this statement. 

Thanks to this maximal inequality, it makes sense to consider the maximal function $\sup_{r>0}\left|M_r^Sf\right|$ for any $f\in L^p(\R^d)$ with $p>\frac{d}{d-1}$, and the maximal inequality (\ref{max-in-s}) remains valid for all $f\in L^p(\R^d)$.

As a direct consequence, we have the following differentiation theorem (which is evident for continuous functions, and which extends by density thanks to the maximal inequality).

\begin{corollary} Let $d\geq2$ and $p>\frac{d}{d-1}$. For all $f\in L^p_{\rm loc}(\R^d)$, for almost all $x\in \R^d$,
$$
\lim_{r\to0}M_r^Sf(x)=f(x).
$$
\end{corollary}

The Calder\`on transference principle (see for example \cite{etga}, Section 2.3) allows us to transfer the maximal inequality (\ref{max-in-s}) to the general context of a measure preserving $\R^d$-action.

Let $(\Omega,\T,\mu)$ be a finite or $\sigma$-finite measure space and $(T_t)_{t\in\R^d}$ be a measurable and measure preserving action of the group $\R^d$ on this space. If $\phi$ is an integrable function on $\Omega$, we denote
$$
A_r^S\phi(\omega):=\int_{S_r}\phi(T_t\omega)\,\text{d}\sigma_r(t).
$$
After transfer, the maximal inequality (\ref{max-in-s}) takes the following form: let $d\geq2$ and $p>\frac{d}{d-1}$; if $\phi\in L^p(\mu)$, the family $\left(A_r^S\phi\right)_{r>0}$ is $\mu$-almost everywhere well defined and
$$
\int_{\Omega}\sup_{r>0}\left|A_r^S\phi(\omega)\right|^p\,\text{d}\mu(\omega)\leq C(d,p)\int_\Omega\left|\phi(\omega)\right|^p\,\text{d}\mu(\omega).
$$

The pointwise ergodic theorem for spherical averages is the following result. Here we suppose that $\mu$ is a probability measure.
\begin{theorem}
Let $d\geq2$, $p>\frac{d}{d-1}$, and $\phi\in L^p(\mu)$. For $\mu$-almost every $\omega$,
$$
\lim_{r\to+\infty}A_r^S\phi(\omega)\ \ \text{exists}.
$$
Moreover, if the $\R^d$-action on $(\Omega,\T,\mu)$ is ergodic, then this limit is $\displaystyle\int_\Omega\phi\,\rm{d}\mu$.
\end{theorem}
The preceding theorem is due to Jones in the case $d\geq3$ and to Lacey in the case $d=2$. Of course, the maximal inequality plays a crucial role in the proof, but we notice that the pointwise ergodic theorem is not a direct consequence of the maximal inequality since the spheres in $\R^d$ do not form a F\o lner family. 

The mean ergodic theorem is much easier to prove. Via the spectral theorem, it is a direct consequence of the fact that the Fourier transform of the measure $\sigma_r$ tends to zero at any non-zero point when $r$ tends to infinity. In order to set out the mean ergodic theorem, notice first that, by Fubini Theorem, for all $p\geq1$ and $\phi\in L^p(\mu)$, for each $r>0$, the function $A_r^S\phi$ is well defined as an element of $L^p(\mu)$. The mean ergodic theorem then states that there exists $\tilde\phi\in L^p(\mu)$, invariant under the transformations $T_t$, such that
$$
\lim_{r\to+\infty}\int_\Omega\left|A_r^S\phi-\tilde\phi\right|^p\,\text{d}\mu=0.
$$
\section{Averages on annuli. Basic facts}\label{facile}
A function $e$, from $(0,+\infty)$ into itself, is given such that $e(r)\leq r$. We consider averaging operators $M_r$ and $A_r$ defined in Introduction. The first thing to make sure is that the objects we want to study are well defined.
Let $(\Omega,\T,\mu)$ be a measure space and $(T_t)_{t\in\R^d}$ be a measurable and measure preserving action of the group $\R^d$ on this space. 
\begin{proposition} \label{def} If $\phi$ is an integrable function on $\Omega$, then, for $\mu$-almost all $\omega$, the function $t\mapsto\phi(T_t\omega)$ is locally integrable on $\R^d$. If the function $\phi$ is equal to zero $\mu$-almost everywhere, then, for $\mu$-almost all $\omega$, $\phi(T_t\omega)=0$ for almost all $t$.
\end{proposition} 
\begin{proof}
Let $\phi \in \mathcal{L}^1(\Omega)$. By Fubini Theorem, for all $R>0$, we have, 
\begin{eqnarray*}
\int_\Omega\left(\int_{B_R}|\phi|(T_t\omega) \, \text{d}t\right) \, \text{d}\mu(\omega)
& = &
\int_{B_R}\left(\int_{\Omega}|\phi|(T_t\omega) \, \text{d}\mu(\omega)\right) \, \text{d}t\\
& = &
|B_R|\int_\Omega|\phi|(\omega) \, \text{d}\mu(\omega).
\end{eqnarray*}
Hence, for all $R >0$, for $\mu$-almost all $\omega$, the function
$t \mapsto \phi(T_t\omega)$ is integrable on $B_R$.
This implies that for $\mu$-almost all $\omega$,
for all $R\in \N$,  the function
$t \mapsto \phi(T_t\omega)$ is integrable on $B_R$, which means that this function is locally integrable.
\\
If $\phi=0$ $\mu$-almost everywhere, the preceding calculus gives:\\
for all $R>0$, for $\mu$-almost all $\omega$, $\phi(T_t\omega)=0$ for almost all $t$ in $B_R$.\\
Once more we exchange the quantifiers and we obtain that for $\mu$-almost all $\omega$, for all $R>0$,
$\phi(T_tx)=0$ for almost all $t$ in $B_R$.
\end{proof}
\begin{corollary} For all $\phi\in L^1(\Omega)$, the family $(A_r\phi)_{r>0}$ is well defined $\mu$-almost everywhere.
\end{corollary}
In particular, if $f$ is a locally integrable function on $\R^d$, then the family $(M_rf)_{r>0}$ is almost everywhere well defined.

\medbreak\medbreak
Now we will see that the mean ergodic theorem for these averages on annuli does not present any difficulty. Suppose that $\mu$ is a probability measure. If $\phi$ belongs to the Hilbert space $L^2(\mu)$, we denote by $\tilde\phi$ its orthogonal projection on the subspace of functions which are invariant under the transformations $T_t$, $t\in\R^d$. This operation extends to $L^1(\mu)$: $\tilde\phi$ is the conditionnal expectation of $\phi$ with respect to the sub-$\sigma$-algebra of $T_t$-invariant events.
\begin{theorem}[Mean ergodic theorem for averages on annuli]\label{met}
Let $d\geq2$. For all $p\geq1$ and for all $\phi\in L^p(\mu)$,
$$
\lim_{r\to+\infty}\int_\Omega\left|A_r\phi-\tilde\phi\right|^p\,\text{d}\mu=0.
$$
\end{theorem}

In the case when the thickness of the annuli tends to infinity with the radius (i.e. when $e(r)\to+\infty$ if $r\to+\infty$), one can show that the family $(C_{r,e(r)})_{r>0}$ has the F\o lner property. In this case, the mean ergodic theorem follows directly from general classical arguments (see for example \cite{etga}, Section 2.2). But in the case when the thickness does not tend to infinity, the F\o lner property is not satisfied anymore and another argument is required. We will use a Fourier transform argument; since it is not original, we only give its main lines.
\begin{proof}[Sketch of the proof of Theorem \ref{met}]
Using spherical coordinates, the Fourier transform
$$
\widehat c_r(z)=\frac1{\left|C_{r,e(r)}\right|}\int_{C_{r,e(r)}}\exp(2\pi iz\cdot x)\,\text{d}x\quad\quad(z\in\R^d)
$$
of the uniform measure on the annulus $C_{r,e(r)}$ can be expressed 
$$
\widehat{c_r}(z):=\frac{k_d}{\left|C_{r,e(r)}\right|}\int_{r-e(r)}^{r}\rho^{d-1}\widehat{\sigma_1}(\rho z)\,\text{d}\rho,
$$
where $\widehat{\sigma_1}$ is the Fourier transform of the uniform probability on the unit sphere and $k_d$ is a constant. 
We know that, for all $z\in\R^d\setminus\{0\}$, $\displaystyle\lim_{\rho\to+\infty}\widehat{\sigma_1}(\rho z)=0$ (see for example (\cite{St2}, page 347 ) or (\cite{etga}, Section 5.1). From this, we easily deduce that, for all non zero $z$, $\displaystyle\lim_{r\to+\infty}\widehat{c_r}(z)=0$.

Consider now $\phi\in L^2(\mu)$. The spectral theorem associates to $\phi$ a positive finite measure $\nu$ on $\R^d$ such that, for all $t\in\R^d$
$$
\int_{\Omega}\phi(T_t\omega)\overline{\phi(\omega)}\,\text{d}\mu(\omega)=\int_{\R^d}\exp(2\pi iz\cdot t)\,\text{d}\nu(z).
$$
Moreover the function $\phi$ is orthogonal to the subspace of $(T_t)_{t\in\R^d}$-invariant functions if and only if $\nu(\{0\})=0$. 

We have
$$
\int_{\Omega}\left|A_r\phi(\omega)\right|^2\,\text{d}\mu(\omega)=\int_{\R^d}\left|\widehat{c_r}(z)\right|^2\ \text{d}\nu(z),
$$
and this last quantity tends to $\nu(\{0\})$ as $r$ goes to infinity, by dominated convergence.

This proves the mean ergodic theorem for $p=2$.

For any $p\geq1$, we conclude along the following lines: \\for all $\phi\in L^{\infty}(\mu)$, $\displaystyle \tilde\phi=\lim_{r\to+\infty}A_r\phi$ in $L^2(\mu)$; \\by dominated convergence, this convergence takes also place in $L^p(\mu)$; \\by density, this convergence extends to any $\phi\in L^p(\mu)$.

\end{proof}

\section{Averages on annuli. Maximal inequality}\label{principal}
\subsection{Introduction}
As in the preceding section the family of annuli $C_r:=C_{r,e(r)}$ is given. Let $p\in[1,+\infty)$. We say that the family of operators $(A_r)_{r>0}$ satisfies the strong-$L^p$ maximal inequality if there exists a constant $C(d,p)$ such that, for any measure preserving system $(\Omega,\T,\mu,(T_t)_{t\in\R^d})$ and for any $\phi\in L^p(\mu)$,
$$
\int_{\Omega}\sup_{r>0}\left|A_r\phi(\omega)\right|^p\,\text{d}\mu(\omega)\leq C(d,p)\int_{\Omega}\left|\phi(\omega)\right|^p\,\text{d}\mu(\omega)\;.
$$
We say that the family of operators $(A_r)_{r>0}$ satisfies the weak-$L^p$ maximal inequality if there exists a constant $C(d,p)$ such that, for any measure preserving system $(\Omega,\T,\mu,(T_t)_{t\in\R^d})$ and for any $\phi\in L^p(\mu)$,
$$
\mu\left\{\omega\in\Omega\mid\sup _{r>0}\left|A_r\phi(\omega)\right|>1\right\}\leq C(d,p)\int_{\Omega}\left|\phi(\omega)\right|^p\,\text{d}\mu(\omega)\;.
$$

\medbreak
In the study of maximal inequalities for averages on annuli, we will keep in mind the following two classical facts.\begin{itemize}\item Thanks to Calder\`on transference principle, the validity of a maximal inequality in the general context of a measure preserving system $(\Omega,\T,\mu,(T_t)_{t\in\R^d})$ will be guaranteed by the validity of this maximal inequality in the particular context of the action of $\R^d$ on itself by translation.
\item Thanks to the general construction of Rokhlin towers, the negation of a maximal inequality in this particular context implies the negation of this maximal inequality in any aperiodic system.
\end{itemize}
\subsection{Direct consequences of known results on spheres and balls}
The average operator on an annulus can always be written as an average of spherical averages:
$$
M_rf(x)=\frac{c_d}{\left|C_{r}\right|}\int_{r-e(r)}^{r}\rho^{d-1}\left(\int_{S_1}f(x+\rho\theta)\,\text{d}\sigma_1(\theta)\right)\,\text{d}\rho\;,
$$
where $c_d$ is a constant.
As a consequence, any maximal inequality for spherical averages gives a maximal inequality for averages on annuli. Thus, the following theorem is a corollary of Theorem \ref{max-in-s}.
\begin{theorem}\label{easy-un}
Let $d \geq 2$. If $p>\frac{d}{d-1}$, then the family of operators $(A_r)_{r>0}$ satisfies the strong-$L^p$ maximal inequality. 
\end{theorem}

For averages on balls, the strong-$L^p$ maximal inequality ($p>1$) and the weak-$L^1$ maximal inequality are known (\cite{w}). Just writing down the trivial inequality that the integral of a non-negative function on $C_r$ is bounded by its integral on the ball $B_r$, we obtain maximal inequalities for averages on annuli, as soon as these annuli have a volume comparable to the volume of the ball. 
\begin{theorem}\label{boule}
If there exists $\gamma>0$ such that, for all $r>0$, $e(r)\geq\gamma r$, then, for all $p > 1$, the family $(A_r)_{r>0}$ satisfies the strong-$L^p$
and the weak-$L^1$ maximal inequalities.
\end{theorem}

In the remainder of Section~3, we suppose that $d\geq2$.

Theorems \ref{easy-un} and \ref{boule} are positive results concerning maximal inequalities. On the other hand, we know that spherical averages do not satisfy any $L^p$ maximal inequality as soon as $p\leq\frac{d}{d-1}$. We could have expected that the range of validity of a maximal inequality would evolve ``continuously'' with the choice of the thickness function $e$. Next theorem shows that it is not the case. It says essentially that, as soon as the function $e$ is reasonably regular and does not satisfy the hypothesis of Theorem \ref{boule}, averages on annuli behave like averages on spheres. 

\subsection{The dichotomy theorem}

\begin{theorem}\label{main}
If one of the following three hypothesis is satisfied
\begin{itemize}\item[(h1)] the function $e$ is non-decreasing and $\displaystyle\inf_{r>0}\frac{e(r)}{r}=0$,
\item[(h2)] $\displaystyle\lim_{r\to+\infty}\frac{e(r)}{r}=0$,
\item[(h3)] $\displaystyle\lim_{r\to0}\frac{e(r)}{r}=0$,
\end{itemize}
then, for all $p\leq\frac{d}{d-1}$, the family of operators $(A_r)_{r>0}$ does not satisfy the weak-$L^p$ maximal inequality. 
\end{theorem}

In fact, we will consider the following hypothesis on the thickness function~$e$: the function $e$ is defined on a sub-interval $I$ of $(0,+\infty)$ and
\begin{equation}\label{hyp}
\exists a>1, \forall\delta>0, \exists h>0, \text{ such that }[h,ah]\subset I \text{ and }\forall r\in[h,ah], e(r)\leq\delta r.
\end{equation}
This is not the optimal hypothesis, but it keeps a useful form. It is not difficult to verify that each of the Hypothesis (h$i$) of the Theorem implies Hypothesis~(\ref{hyp}).

We will prove that, under Hypothesis (\ref{hyp}), there is no weak-$L^{d/(d-1)}$ maximal inequality. Following the same construction or using a classical interpolation argument, the result extends to any $L^p$, with $p\leq\frac{d}{d-1}$.

 Let us begin by a geometrical lemma.\begin{lemma}\label{evid-geom}
Let $\mathcal{C}$ be an annulus of thickness $\epsilon \in ]0,1]$
with center at $x$ and external radius $\|x\| > 1$.
If $\rho \in [\epsilon,1]$ then $\sigma_\rho(S_{\rho}
\cap \mathcal{C}) \geq c_d \epsilon/ \rho$.
\end{lemma}

In dimension 2, this lemma says that the length of an arc defined as the intersection of the annulus $C_{\|x\|,\epsilon}$ with a circle of radius $\rho$ centered on the exterior boundary of the annulus is greater than $\epsilon$. This can be considered as a geometrical evidence. However, we propose a proof, given here with details in dimension 3.
\begin{proof}
We choose a Cartesian system of coordinates such that point $x$ has coordinates $(0,0,\|x\|)$, and we also use spherical coordinates $(\rho,\phi,\alpha)$, $\rho>0$, $0<\phi<2\pi$, $0<\alpha<\pi$. We have
\begin{eqnarray*}
{\rm Area}(S_{\rho}
\cap \mathcal{C}) 
& = &
\int_{\{\theta \in S(0,1) \mid \|x\| -\epsilon \leq
  \|\rho \theta - x\| \leq \|x\|\}}\rho^2 \, \text{d}\sigma(\theta)\\
& = &
\int_0^{2\pi}\int_0^{\pi} {\bf 1}_I(\rho,\phi,\alpha)\rho^2\sin\alpha \, \text{d}\alpha \,
\text{d}\phi,
\end{eqnarray*}
where $I$ is the integration domain.\\
We see that
$$I=\left\{(\rho,\phi,\alpha) \mid
\frac{\rho^2}{2\|x\|}\leq \rho\cos\alpha\leq \frac{2\epsilon\|x\|+\rho^2-\epsilon^2}{2\|x\|}\right\}.$$
With the change of variable $u=\rho\cos\alpha$, we can write
\begin{eqnarray*}
{\rm Area}(S_{\rho}
\cap \mathcal{C}) 
& = & 2 \pi \rho \int_{-\rho}^{\rho} {\bf 1}_{\frac{\rho^2}{2\|x\|}\leq u
  \leq 
\frac{2\epsilon\|x\|+\rho^2-\epsilon^2}{2\|x\|}}(u)
\, \text{d}u\\
& = &
 2 \pi \rho\ \lambda\left([-\rho,\rho] \cap
\left[\frac{\rho^2}{2\|x\|},\frac{2\epsilon\|x\|+\rho^2-\epsilon^2}{2\|x\|}\right]\right)\\
& = &
 2 \pi \rho\ \min\left(\rho-\frac{\rho^2}{2\|x\|},\frac{2\epsilon\|x\|+\rho^2-\epsilon^2}{2\|x\|}-\frac{\rho^2}{2\|x\|}\right)\\
& = &
 2 \pi
 \rho\left(\frac{2\epsilon\|x\|+\rho^2-\epsilon^2}{2\|x\|}-\frac{\rho^2}{2\|x\|}\right) 
\mbox{ , since } \|x\| > 1 \geq \rho \geq \epsilon,\\
& \geq &
c_3 \rho \epsilon.
\end{eqnarray*}

\end{proof}

\begin{lemma}\label{func-crit}
There exists a non-negative function $f$ on $\R^d$ such that
\begin{equation}\label{int-as}
\int_{\R^d} f(x)^{d/(d-1)}\,\text{d}x<\infty,
\end{equation}
and, if $\|x\|>1$ and $e(\|x\|)\leq\frac14$, then
$$
M_{\|x\|}f(x)\geq\frac{\ln|\ln e(\|x\|)|}{\|x\|^{d-1}}.
$$
\end{lemma}
\begin{proof}
We consider a non-negative function $f$ such that $$
   f(x)=   
                          -\frac{1}{\|x\|^{d-1}\ln(\|x\|)} \quad
                            \mbox{ si }
                          x \in  B_{1/2} .
                          $$
Outside the ball $B_{1/2}$ the function $f$ can be chosen to be zero. It is also possible to choose $f$ to be smooth (except at zero) and with bounded support. In any case the integrability assumption (\ref{int-as}) is satisfied.

For all $x \in \R^d$, we have 
\begin{eqnarray*}
M_{\|x\|}f(x)
& \geq &
-c_d\frac{1}{\|x\|^{d-1}e(\|x\|)}\int_{
                       \substack{\|x\|-e(\|x\|) \leq \|t\| \leq \|x\| \\
                          \|x+t\| \leq 1/2}}\frac{\text{d}t}{\|x+t\|^{d-1} \ln(\|x+t\|)}\\
& \geq &
-c_d\frac{1}{\|x\|^{d-1}e(\|x\|)}\int_{
                       \substack{\|x\|-e(\|x\|) \leq \|t|| \leq \|x\| \\
                          e(\|x\|) \leq \|x+t\| \leq 1/2}}
\frac{\text{d}t}{\|x+t\|^{d-1} \ln(\|x+t\|)}\\
& \geq &
-c_d\frac{1}{\|x\|^{d-1}e(\|x\|)}\int_{
                       \substack{\|x\|-e(\|x\|) \leq \|t-x\| \leq \|x\| \\
                          e(\|x\|) \leq \|t\| \leq 1/2}}
\frac{\text{d}t}{\|t\|^{d-1} \ln(\|t\|)}.
\end{eqnarray*}
Using spherical coordinates, we can write \\
\begin{multline*}\int_{
\substack{\|x\|-e(\|x\|) \leq \|t-x\| \leq \|x\|\\ 
e(\|x\|) \leq \|t\| \leq 1/2}}
\frac{\text{d}t}{\|t\|^{d-1} \ln(\|t\|)}\\
 = 
c_d \int_{e(\|x\|)}^{1/2}\frac{1}{\rho^{d-1} \ln(\rho)}
\left( \int_{\{\theta \in S_1 \mid \|x\|-e(\|x\|) \leq \|\rho\theta - x\| \leq \|x\|\}}
\rho^{d-1} \, \text{d}\sigma(\theta)\right) \, \text{d}\rho\\
 \geq  
c_d \int_{e(\|x\|)}^{1/2}\frac{1}{\rho \ln(\rho)}e(\|x\|) \,
\text{d}\rho \; ,
\end{multline*}
the last identity beeing a consequence of Lemma \ref{evid-geom} (since we suppose that $\|x\|>1$).
We obtain
\begin{eqnarray*}
M_{\|x\|}f(x)
& \geq &
-c_d\frac{1}{\|x\|^{d-1}e(\|x\|)}e(\|x\|)\int_{
                       e(\|x\|) \leq \rho \leq 1/2
                      }\frac{d\rho}{\rho \ln(\rho)}\\
& \geq &
c_d\frac{1}{\|x\|^{d-1}}(\ln|\ln e(\|x\|)|-\ln|\ln(1/2)|)\\
& \geq &
\frac{1}{2}c_d\frac{1}{\|x\|^{d-1}}\ln|\ln e(\|x\|)| \, , 
\mbox{ since }e(\|x\|)\leq1/4.
\end{eqnarray*}
In order to obtain the function announced in Lemma \ref{func-crit} we just have to multiply this function $f$ by a well chosen real constant.
\end{proof}
We know fix a function $f$ satisfying the properties presented in Lemma \ref{func-crit}. 
\begin{lemma} \label{cour}Let $a>1$ and $0<\delta<\frac14$ be given such that, for all $r\in[1,a]$, $e(r)\leq\delta$. 

There exists $\lambda>0$ such that
$$
\lambda^{d/(d-1)}\left|\left\{x\in\R^d\mid M_{\|x\|}f(x)\geq\lambda\right\}\right|\geq c_d(1-a^{-d})(\ln|\ln\delta|)^{d/(d-1)}.
$$
\end{lemma}

\begin{proof}
From Lemma \ref{func-crit}, we deduce that, for all $\lambda>0$, the set of points $x$ such that $M_{\|x\|}f(x)\geq\lambda$ contains the annulus defined by $\|x\|\in(1,a]$ and $\displaystyle  \frac{\ln|\ln e(\|x\|)|}{\|x\|^{d-1}}\geq\lambda$. Hence this set contains the annulus defined by $\|x\|\in(1,a]$ and $\displaystyle  \frac{\ln|\ln\delta|}{\|x\|^{d-1}}\geq\lambda$.
\\
We choose $ \displaystyle\lambda=\frac{\ln|\ln\delta|}{a^{d-1}}$, so that the set of points $x$ such that $M_{\|x\|}f(x)\geq\lambda$ contains the annulus defined by $\|x\|\in(1,a]$.
We have 
$$
\left|\left\{x\in\R^d\mid M_{\|x\|}f(x)\geq\lambda\right\}\right|\geq c_d(a^d-1),
$$
and multiplying each term by $\lambda^{d/(d-1)}$, we obtain the inequality stated in Lemma \ref{cour}.
\end{proof}
\begin{proof}[End of the proof of Theorem \ref{main}]
Now suppose that Hypothesis (\ref{hyp}) is satisfied. Let $\delta\in(0,\frac14)$. We fix a number $h>0$ such that $e(r)\leq\delta r/a$ for all $r\in[h,ah]$. We consider the function $g$ defined by $g(x)=f(\frac1hx)$. The mean value of $g$ on the annulus of center $x$, radius $\|x\|$ and thickness $e(\|x\|)$ is the mean value of $f$ on the annulus of center $\frac1hx$, radius $\|x\|/h$ and thickness $e(\|x\|)/h$. We have $e(r)/h\leq\delta$ when $r/h\in[1,a]$. From Lemma \ref{cour}, we deduce that there exists $\lambda>0$ such that 
$$
\lambda^{d/(d-1)}\left|\left\{\frac1hx\mid M_{\|x\|}g(x)\geq\lambda\right\}\right|\geq c_d(1-a^{-d})(\ln|\ln\delta|)^{d/(d-1)}.
$$
Note also that $\int_{\R^d}g(x)^{d/(d-1)}\,\text{d}x=h^d\int_{\R^d}f(x)^{d/(d-1)}\,\text{d}x$. This last integral is a given constant denoted by $C$. We conclude that
$$
\lambda^{d/(d-1)}\left|\left\{x\mid M_{\|x\|}g(x)\geq\lambda\right\}\right|\geq \frac{c_d}{C}(1-a^{-d})(\ln|\ln\delta|)^{d/(d-1)}\int_{\R^d}g(x)^{d/(d-1)}\,\text{d}x.
$$
Since for arbitrarily small $\delta$ we can find such a function $g$, we conclude that the weak-$L^{d/(d-1)}$ is not satisfied, and the theorem is proved.
\end{proof}
\subsection{How to get rid of the dichotomy} The conclusion of the preceding study is that if the thickness function $e$ is regular enough, the critical exponent for the maximal inequality is either $d/(d-1)$ or 1. It is however possible to get rid of this dichotomy, by using results concerning spherical averages $A^S_r$ when the radius $r$ is only considered in a sparse subset of the real line. We will refer here to the article by Seeger, Wainger and Wright (\cite{sww}). 

Let $p_0$ be any fixed number in the interval $(1,\frac{d}{d-1})$.
Following \cite{sww}, there exists a sequence of positive numbers $(r_n)_{n\in\N}$ such that the family of average operators $\left(A^S_{r_n}\right)_{n\in\N}$ satisfies the strong-$L^p$ maximal inequality for all $p>p_0$ and does not satisfy the strong-$L^p$ maximal inequality for any $p<p_0$. 

Coming back to our setting, this suggests the following construction. Let us consider a sequence of positive reals $(\e_n)$ and define the thickness function $e$ by $e(r_n)=\e_n$, and $e(r)=r$ for the values of $r$ that do not appear in the sequence $(r_n)$. If the sequence $(\e_n)$ tends to zero rapidly enough, we will obtain that: for all $p>p_0$, the family $(A_r)_{r>0}$ satisfies the strong-$L^p$ maximal inequality; for all $p<p_0$, the family $(A_r)_{r>0}$ does not satisfy the strong-$L^p$ maximal inequality.

\section{Averages on annuli. Ergodic and differentiation theorems.}\label{convergence}

It is well known that a maximal inequality is necessary in order to have a differentiation theorem or a pointwise ergodic theorem for a family of averages. This fact is a consequence of the \emph{Banach principle}, described, for example in \cite{garsia} and \cite{etga}.

On the other hand, the differentiation property
$$
\lim_{r\to0}M_rf(x)=f(x) \quad\text{for a.e. $x$}
$$
is valid (for any $f\in L^p(\R^d)$) as soon as a $L^p$ maximal inequality is satisfied.

And writing averages on annuli as averages of averages on spheres, we see easily that the ergodic theorem
$$
\lim_{r\to+\infty}A_r\phi(\omega) \quad\text{exists for $\mu$-almost all $\omega$}
$$
(and equals the integral of $\phi$ when the action is ergodic)
is valid as soon as it is valid for spherical averages. 

Thus we have enough information in the preceding sections to conclude that in the case when $e(r)/r$ is bounded from below the differentiation theorem and the ergodic theorem for averages on annuli are satisfied under the sole integrability assumption, and in the case when $d\geq2$ and $e(r)/r$ tends to zero at zero (resp. at infinity) the  differentiation theorem (resp. the pointwise ergodic theorem) is valid under the integrability assumption $p>d/(d-1)$ and invalid under the integrability assumption $p\leq d/(d-1)$.

Let us be a little more precise on the meaning of the adjective ``invalid". Saying that the differentiation theorem is invalid for $p\leq d/(d-1)$ means that there exists a function $f$ on $\R^d$, which is locally of power $d/(d-1)$ integrable and for which $\lim_{r\to 0}M_rf$ does not exist in the almost everywhere sense (of course, it exists and equal $f$ in the $L^{d/(d-1)}_{loc}$ sense). Saying that the pointwise ergodic theorem is invalid for $p\leq d/(d-1)$ means that, for any aperiodic probability measure preserving dynamical system $\left(\Omega,\T,\mu,(T_t)_{t\in\R^d}\right)$, there exists $\phi\in L^{d/(d-1)}(\mu)$ such that the ergodic averages $A_r\phi$ are not almost everywhere convergent when $r$ tends to infinity. (Recall that they converge in the mean, by Theorem \ref{met}.)

\subsection{Remark on non-spherical annuli}
The facts that are presented in Sections \ref{principal} and \ref{convergence} concern spherical annuli. They can certainly be extended to other geometrical domains; indeed, it is known that the maximal inequality and differentiation theorem for spherical averages can be extended to averages on homotetic images of very general hypersurfaces, and there is no difficulty to define an associated notion of annuli. However, these hypersurfaces have to satisfy some curvature conditions in order, roughly speaking, to avoid any ``flat part''. These facts are very well described in Stein's book \cite{St2}, Section XI.3.

We will not enter in this level of generality in the present article, but we propose to state, as a typical example, that ``cubic annuli'' are in general bad for differentiation and ergodic theorems.

Mimicing the notation we used for spherical annuli, we define
$$
D_{r,e}:=\{x\in\R^d\mid r-e\leq\|x\|_\infty\leq r\},
$$
where $\|(x_1,x_2,\ldots,x_d)\|_\infty=\max_{1\leq i\leq d}|x_i|$.

Of course Hardy-Littlewood maximal inequality and Wiener ergodic theorem are valid for $\|\cdot\|_\infty$-``balls'' as well as they are valid for Euclidean balls. Thus, if we consider a thickness function $r\mapsto e(r)$ such that $e(r)/r$ stays bounded by below, we have an analogue of Theorem \ref{boule} and differentiation theorem, as well are pointwise ergodic theorem, are true for averages on $D_{r,e(r)}$ for the class of integrable functions.

If in contrast $\inf_{r>0}e(r)/r=0$  the situation is radically different. Using negative results on the maximal inequality that are known in dimension~1, it is possible to prove that if $e$ is non-decreasing and $\inf_{r>0}e(r)/r=0$ then the averages on $D_{r,e(r)}$ do not satisfy any $L^p$ maximal inequality. If $\lim_{r\to0}e(r)/r=0$, there is no hope to obtain a differentiation theorem for averages on cubic annuli, even in the class of bounded measurable functions $f$. If $\lim_{r\to+\infty}e(r)/r=0$, there is no hope to obtain a pointwise ergodic theorem for averages on cubic annuli, even in the class of bounded measurable functions $\phi$.

\section{Dimension 1}\label{dimun}

In dimension 1, our averages take the following form :
$$
M_rf(x)=\frac1{2e(r)}\left(\int_{-r}^{-r+e(r)}f(x+t)\,\text{d}t+\int_{r-e(r)}^{r}f(x+t)\,\text{d}t\right)
$$
and
$$
A_r\phi(\omega)=\frac1{2e(r)}\left(\int_{-r}^{-r+e(r)}\phi(T_t\omega)\,\text{d}t+\int_{r-e(r)}^{r}\phi(T_t\omega)\,\text{d}t\right).
$$

They are similar to ``moving ergodic averages'' studied in the discrete time case in \cite{Bellow}.

We have the following facts.
\begin{itemize}
\item[(i)]
The mean ergodic theorem (cf Theorem \ref{met}) is satisfied by the averages $(A_r)_{r>0}$ if and only if $\lim_{r\to\infty}e(r)=+\infty$.
\item[(ii)]
If $\inf_{r>0}e(r)/r>0$, then a weak-$L^1$ maximal inequality and a strong-$L^p$ maximal inequality are satisfied by these averages, for all $p\in(1,+\infty)$.
\item[(iii)]
If one of Hypothesis (h1), (h2) or (h3) of Theorem \ref{main} is satisfied, none of the preceding maximal inequality is satisfied.
\end{itemize}

The ``if'' part of fact (i) is a direct consequence of the computation of the Fourier transform of the kernel of $M_r$; it can also be verified easily that the F\o lner property is satisfied. The ``only if'' part of fact (i) can be verified by an explicit construction, using Rokhlin towers.

Fact (ii) follows directly from the classical Hardy-Littlewood maximal inequality in dimension 1.

Fact (iii) can be verified by using a cone condition similar to the condition described in \cite{Bellow}. We do not give here the details of this argument, which is described in the doctorat thesis of the first author (\cite{fh}).

\end{document}